\documentclass[12pt]{article}
\usepackage{graphicx}
\usepackage{subfigure}
\usepackage{amsthm,amsfonts}
\usepackage{amssymb}
\usepackage{fullpage}
\usepackage{caption}
\usepackage{amsmath}
\usepackage{pifont}
\usepackage{hyperref}
\newtheorem{definition}{Definition}[section]
\newtheorem{theorem}[definition]{Theorem}
\newtheorem{lemma}[definition]{Lemma}
\newtheorem{corollary}[definition]{Corollary}
\newtheorem{proposition}[definition]{Proposition}
\newtheorem{question}[definition]{Question}
\def \N {\mathbb N}

\begin{document}
\title{The nonrepetitive coloring of grids}
\author{Tianyi Tao \\
School of mathematical sciences \\
Fudan University \\
{\tt tytao20@fudan.edu.cn}}
\date{\today}
\maketitle
	
\begin{abstract}
For a graph $G$, a vertex coloring $f$ is called nonrepetitive if for all $k\in\N$ and all $P_{2k}=\langle v_1,\cdots,v_k,v_{k+1},\cdots,v_{2k}\rangle$ (path of $2k$ vertices) in $G$, there must be some $1\le i\le k$ such that $f(v_i)\not=f(v_{k+i})$.

We use $\pi(G)$ to denote the minimum number of colors required for $G$ to be nonrepetitively colored.
		
In 1906, Thue proved that $\pi(P_n)\le3$ for all $n$. In this paper, we focus on grids, which are the Cartesian products of paths. We prove that $5\le\pi(P_n\square P_n)\le12$ for sufficiently large $n$, where the previous best lower bound was 4 and upper bound was 16. Moreover, we also discuss nonrepetitive coloring of the Cartesian product of complete graphs.
\end{abstract}
	
\section{Introduction}
First, we standardize the notation used in this paper. We use the uppercase letters $A$, $B$, $C$, $D$, $X$, $Y$, $Z$, $W$ and the lowercase letters $a$, $b$, $c$, $d$, $x$, $y$, $z$, $w$ to denote colors, the lowercase letter $f$ to denote a coloring function of a graph, the uppercase letter $F$ to denote a function defined on integers, the uppercase letter $P$ to denote a path, the uppercase letter $W$ to denote a walk or lazy walk, and the lowercase letters $u$ and $v$ to denote vertices of a graph (notations may have superscripts or subscripts). In section 4, we use positive integers to denote colors.

In this paper, a $gird$ is the $Cartesian$ $product$ of two paths. For each integer $s\ge2$, an $s$-dimensional grid is the Cartesian product of $s$ paths.

\begin{definition}
Let $G$ be a graph. Suppose $f$ is a vertex coloring of $G$, $k$ is a positive integer and $P_{2k}=\langle v_1,\cdots,v_k,v_{k+1},\cdots,v_{2k}\rangle$ is a path of $2k$ vertices in $G$. Then, we say that $f$ repetitively colors $P_{2k}$ if for all $1\le i\le k,$ $f(v_i)=f(v_{k+i})$. If no even-vertex path of $G$ is repetitively colored by $f$, then we say that $f$ is a (path-)nonrepetitive coloring of $G$. We use $\pi(G)$ to denote the minimum number of colors required for $G$ to be nonrepetitively colored, which is referred to as the nonrepetitive chromatic number of G.
\end{definition}

It is easy to see that a nonrepetitive coloring is a proper coloring.

The origin of studying nonrepetitive coloring can be traced back to the theorem proven by Thue in 1906, with a notable new proof provided by Leech in 1957, which is stated as follows:

\begin{theorem}[\cite{leech19572726,thue1906uber}]
$\pi(P_n)=3$ for all $n\ge4$.
\end{theorem}

For each $n\in\N$, we define a function $F_{P,n}:[n]\to\{a,b,c\}$ such that when each vertex $v_i$ in a path $P_n=\langle v_1,\cdots,v_n\rangle$ receive the color $F_{P,n}(i)$, $P_n$ is nonrepetitively colored. The function $F_{P,n}$ exists due to the above theorem.

Figure 1 provides an example of a nonrepetitive coloring of a path.

\begin{figure}[h] 
	\centering
	\includegraphics[width=0.8\linewidth]{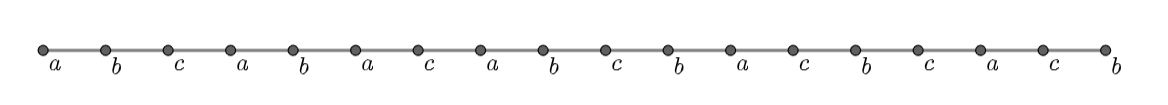}
	\caption{}
\end{figure}

The concept of nonrepetitive coloring of general graphs was introduced by Alon et al.\cite{alon2002nonrepetitive} in 2002. In this paper, they proved that $\pi(G)=O(\Delta^2)$ where $\Delta$ is the maximum degree of $G$. Several authors subsequently improved this result, see \cite{dujmovic2016nonrepetitive,grytczuk2007nonrepetitive,harant2012nonrepetitive,kolipaka2012sharper,RosenfeldMatthieu2020AAtN}. In 2020, Wood et al.\cite{dujmovicVida2020Pghb,dujmovic2020planar,kundgen2008nonrepetitive} proved $\pi(G)\le768$ if $G$ is a planar graph. Additionally, there are some related work in \cite{barat2007square,barat2005NoNG,bose2017new,brevsar2007nonrepetitive,brevsar2004square,fiorenzi2011thue,grytczuk2013new,WoodDavidR2021NGC}.

In \cite{dujmovicVida2020Pghb}, the core method is proving that planar graphs can be embedded into some product structure of several graph classes with bounded nonrepetitive chromatic number, which is also one of the motivations for studying grids in this paper.

We now define the two types of products that will be used in this paper.

\begin{definition}
Given graphs $G_1$ and $G_2$, the Cartesian product of $G_1$ and $G_2$, denoted by $G_1\square G_2$, is the graph with vertex set $V(G_1)\times V(G_2)$, where $(v_1,v_2)$ and $(v_1',v_2')$ are adjacent if $v_1=v_1'$ and $v_2v_2'\in E(G_2)$ or $v_2=v_2'$ and $v_1v_1'\in E(G_1)$; Similarly, the strong product of $G_1$ and $G_2$, denoted by $G_1\boxtimes G_2$, is the graph with vertex set $V(G_1)\times V(G_2)$, where $(v_1,v_2)$ and $(v_1',v_2')$ are adjacent if $(v_1,v_2)$ and $(v_1',v_2')$ are adjacent in $G_1\square G_2$, or $v_1v_1'\in E(G_1)$ and $v_2v_2'\in E(G_2)$.
\end{definition}

It is easy to see that $G_1\square G_2$ is a subgraph of $G_1\boxtimes G_2$.

For our approach, we need to introduce the following concepts:

\begin{definition}
For $k\in\N$, a walk $W_{2k}=\langle v_1,\cdots,v_k,v_{k+1},\cdots,v_{2k}\rangle$ is boring if $v_i=v_{k+i}$ for all $1\le i\le k$.

Let $G$ be a graph, and let $f$ be a vertex coloring of $G$. We say $f$ is walk-nonrepetitive if for all $k$, every repetitively colored walk $W_{2k}$ is boring. That is, $f(v_i)=f(v_{k+i})$ for all $i$ implies $v_i=v_{k+i}$ for all $i$.

We use $\sigma(G)$ to denote the minimum number of colors required for $G$ to be walk-nonrepetitively colored.
\end{definition}

Note that a path is always a walk, so $\pi(G)\le\sigma(G).$

\begin{definition}
A lazy walk in a graph $G$ is a walk in the pseudograph obtained from $G$ by adding a
loop at each vertex. In other words, in a lazy walk, two consecutive vertices may be the same.
\end{definition}

We can also define the concept of lazy walk-nonrepetitive coloring. However, it is easy to check this is an equivalent concept to walk-nonrepetitive coloring. In other words, if $f$ is a walk-nonrepetitive coloring of $G$, then each repetitively colored lazy walk is boring.

Now, we make a simple observation. In a walk-nonrepetitive coloring of $G$, for each $v\in V(G)$ and two neighbors $u_1,u_2$ of $v$, $u_1$ and $u_2$ must have different colors. So it is easy to check that $\sigma(P_6)>3$.

\begin{theorem}[\cite{barat2005NoNG,kundgen2008nonrepetitive}]
$\sigma(P_n)=4$ for all $n\ge6.$
\end{theorem}

\begin{proof}
Let $F_{P,n}:[n]\to\{a,b,c\}$ be the function in Theorem 1.2, we now define a function $f_{W,n}:[n]\to\{a,b,c,d\}$ as follows:

\[F_{W,n}(i)=
\left\{
\begin{array}{lcl}
	d,&i\equiv 0\pmod 3,\\
	F_{P,n}(i-\lfloor\frac{i}{3}\rfloor),&i\not\equiv 0\pmod 3.\\
\end{array}
\right.
\]

For a path $P_n=\langle v_1,\cdots,v_n\rangle$, let $f(v_i)=F_{W,n}(i)$, then $f$ is a walk-nonrepetitive coloring of $P_n$.
\end{proof}

Figure 2 provides an example of a walk-nonrepetitive coloring of a path.

\begin{figure}[h] 
	\centering
	\includegraphics[width=0.8\linewidth]{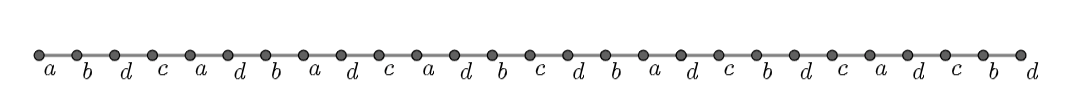}
	\caption{}
\end{figure}

\begin{proposition}[\cite{barat2007square}]
{\rm (i)} If $G'$ is a subgraph of $G$, then $\pi(G')\le\pi(G)$, $\sigma(G')\le\sigma(G)$;

{\rm (ii)} Let $G_1$ and $G_2$ be graphs, then $\sigma(G_1\boxtimes G_2)\le\sigma(G_1)\sigma(G_2)$.
\end{proposition}

\begin{proof}
(i) It is obvious.

(ii) Let $f_1$ and $f_2$ be walk-nonrepetitive coloring on $G_1$ and $G_2$, respectively. We define the following $product$ $coloring$ $f=(f_1,f_2)$ on $G_1\boxtimes G_2$:
\[f((v_1,v_2))=(f_1(v_1),f_2(v_2)).\]
It is not difficult to verify that $f$ is a walk-nonrepetitive coloring of $G_1\boxtimes G_2$. The details of the proof can be found in \cite{barat2007square} or \cite{WoodDavidR2021NGC}.
\end{proof}

\begin{corollary}[\cite{kundgen2008nonrepetitive}]
For all $n$, $\pi(P_n\square P_n)\le\sigma(P_n\square P_n)\le\sigma(P_n\boxtimes P_n)\le\sigma(P_n)^2\le16.$
\end{corollary}

Figure 3 is an example of a nonrepetitive 16-coloring of a grid.

\begin{figure}[h] 
	\centering
	\begin{minipage}[t]{0.48\textwidth}
		\centering
		\includegraphics[width=\textwidth]{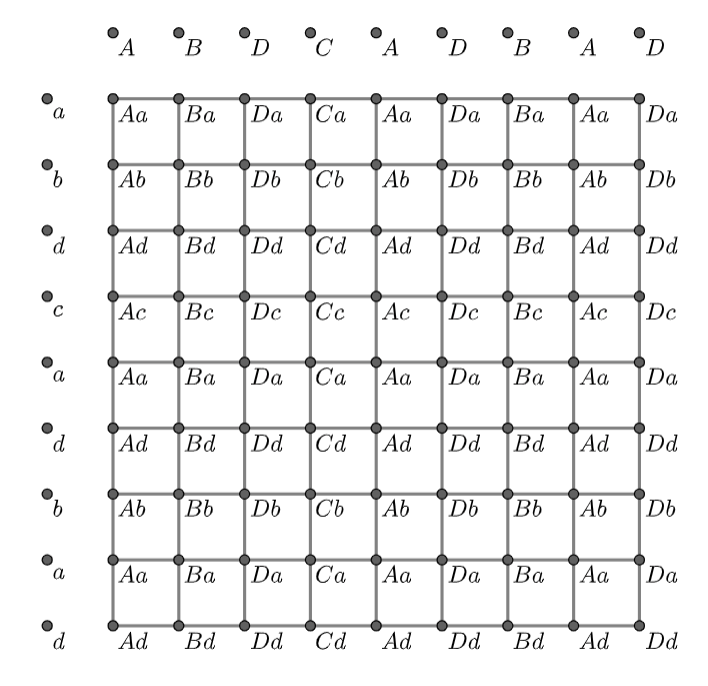}
		\caption{}
	\end{minipage}\hfill
	\begin{minipage}[t]{0.48\textwidth}
		\centering
		\includegraphics[width=\textwidth]{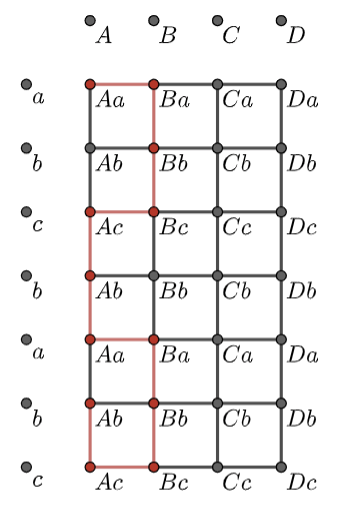}
		\caption{}
	\end{minipage}
\end{figure}

We now consider a natural question. Suppose we assign a product coloring $f=(f_1,f_2)$ to a grid, where $f_1$ is a nonrepetitive 3-coloring of a path, and $f_2$ is a walk-nonrepetitive 4-coloring of a path. Then, is $f$ a nonrepetitive coloring of the grid? Unfortunately, this is incorrect. Figure 4 provides a counterexample, where the repetitive path is highlighted in red. Furthermore, regardless of how a nonrepetitive 3-coloring is designed, there will always exist repetitive $strolls$, where a walk $\langle v_1,\cdots.v_k,v_{k+1},\cdots,v_{2k}\rangle$ is called a stroll if $v_i\not=v_{k+i}$ for all $1\le i\le k$, which indicates that when studying nonrepetitive colorings of grid, product colorings do not yield results better than using 16 colors (see Proposition 3.3 of \cite{WoodDavidR2021NGC}).

In Section 2, we give a coloring to show that $\pi(P_n\square P_n)\le12$. In Section 3, we show that $\pi(P_n\square P_n)\ge5$. In Section 4, we conduct some explorations into the nonrepetitive colorings of Cartesian products of complete graphs. Finally, in Section 5, we pose some open questions.

\section{Upper bound}
\begin{definition}
Label the vertices of the grid $P_n\square P_n$ sequentially from $v_{(1,1)}$ to $v_{(n,n)}$, where $v_{(i,j)}\sim v_{(i',j')}$ if and only if $i=i'$  and $|j-j'|=1$ or $j=j'$ and $|i-i'|=1$. Define the location function $l_1,l_2:V(P_n\square P_n)\to[n]$ as follows: 
\[l_1(v_{(i,j)})=i,\quad l_2(v_{(i,j)})=j.\]

For $2\le r\le2n$, we call the set $L_r=\{v\in V(P_n\square P_n):l_1(v)+l_2(v)=r\}$ a left line.

For $1-n\le r\le n-1$, we call the set $R_r=\{v\in V(P_n\square P_n):l_1(v)-l_2(v)=r\}$ a right line.
\end{definition}

\begin{lemma}
Let $m\ge 2n$. Using the notations in Definition 2.1 and the proof of Theorem 1.6, we construct a vertex coloring $f$ of $P_n\square P_n$ as follows: For all $2\le r\le2n$, let $f(v)=F_{W,m}(r)$ if $v\in L_r$ (Figure 5). Then every repetitive lazy walk $W_{2k}=\langle v_1,\cdots,v_k,v_{k+1},\cdots,v_{2k}\rangle$ satisfies that $v_1$ and $v_{k+1}$ are on the same left line, where $k$ is an arbitrary positive integer.
\end{lemma}

\begin{proof}
Construct a path $P$ as follows: $V(P)=\{L_r:2\le r\le2n\}$, and $L_r\sim L_{r+1}$ for all $r$. In other words, $P$ is obtained by contracting each left line of $P_n \square P_n$ into a single vertex. $P$ inherits a walk-nonrepetitive coloring. $W_{2k}$ corresponds to a repetitive lazy walk in $P$, hence it is boring. As a result, $v_1$ and $v_{k+1}$ lie on the same left line in $P_n\square P_n$.  
\end{proof}

\begin{figure}[h] 
	\centering
	\begin{minipage}[t]{0.48\textwidth}
		\centering
		\includegraphics[width=\textwidth]{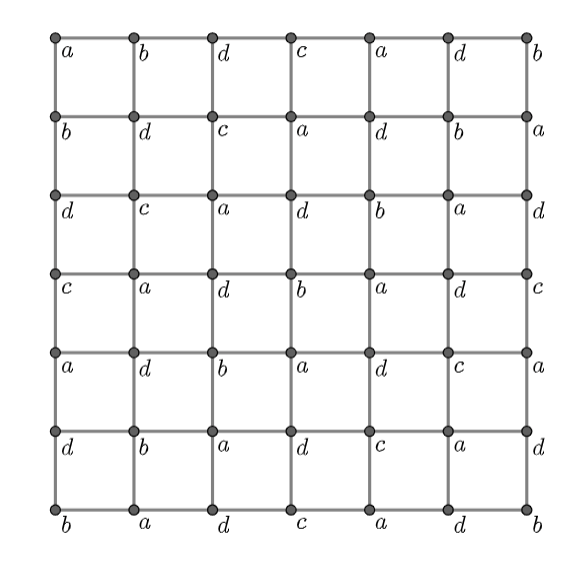}
		\caption{}
	\end{minipage}\hfill
	\begin{minipage}[t]{0.48\textwidth}
		\centering
		\includegraphics[width=\textwidth]{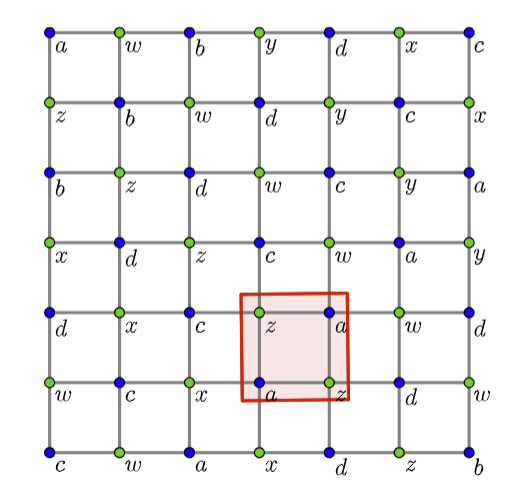}
		\caption{}
	\end{minipage}
\end{figure}

\begin{definition}
For a graph $G$ and a positive integer $s$, define the distance-$s$ graph $G^s$ as follows: $V(G^s)=V(G)$, and $u\sim v$ in $G^s$ if the distance between $u$ and $v$ in $G$ is exactly $s$.
\end{definition}

\begin{theorem}
$\pi(P_n\square P_n)\le 12$ for all $n$.
\end{theorem}

\begin{proof}
Let $m$ be a sufficiently large positive integer. Consider $f_{W,m}^{(1)}:[m]\to\{a,b,c,d\}$ and $f_{W,m}^{(2)}:[m]\to\{x,y,z,w\}$ in the proof of Theorem 1.6. We define a coloring $f_0$ of $P_n\square P_n$ as follows: 

\[f_0(v)=
\left\{
\begin{array}{lcl}
	f_{W,m}^{(1)}(\frac{r}{2}),&v\in L_r, r\equiv 0\pmod 2,\\
	f_{W,m}^{(2)}(\lfloor\frac{r+n+1}{2}\rfloor),&v\in R_r, r\not\equiv 0\pmod 2.\\
\end{array}
\right.
\]

Figure 6 is a schematic diagram of $f_0$. We can observe that $f_0$ is not a nonrepetitive coloring, as there exists a repetitive path of 4 vertices in the red range. Our goal is to prove that if we can destroy all repetitive 4-paths like this, then there will be no more repetitive paths in the graph.

For all $k$, let $P_{2k}=\langle v_1,\cdots,v_k,v_{k+1},\cdots,v_{2k}\rangle$ be a repetitive path in $P_n\square P_n$ under the coloring $f_0$. It is easy to see that $k$ must be even. Let
\[P_{\rm odd}=\{v_1,v_3,\cdots,v_{k-1},v_{k+1},v_{k+3},\cdots,v_{2k-1}\},\]
\[P_{\rm even}=\{v_2,v_4,\cdots,v_{k},v_{k+2},v_{k+4},\cdots,v_{2k}\}.\]
Without loss of generality, we assume that $l_1(v_1)+l_2(v_1)$ is even, which indicates that the colors of vertices in $P_{\rm odd}$ are chosen from $\{a,b,c,d\}$, and the colors of vertices in $P_{\rm even}$ are chosen from $\{x,y,z,w\}$. Note that $P_{\rm odd}$ and $P_{\rm even}$ are both repetitive lazy walks in $(P_n\square P_n)^2$. By Lemma 2.2, $v_1$ and $v_{k+1}$ are in the same left line, $v_2$ and $v_{k+2}$ are in the same right line. This indicates that $v_1\sim v_{k+2}$ and $v_2\sim v_{k+1}$ (Figure 7).

\begin{figure}[h] 
	\centering
	\begin{minipage}[t]{0.48\textwidth}
		\centering
		\includegraphics[width=\textwidth]{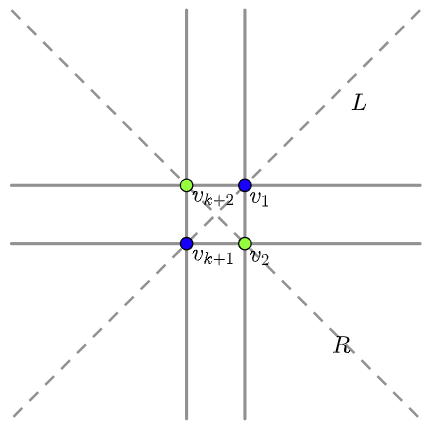}
		\caption{}
	\end{minipage}\hfill
	\begin{minipage}[t]{0.48\textwidth}
		\centering
		\includegraphics[width=\textwidth]{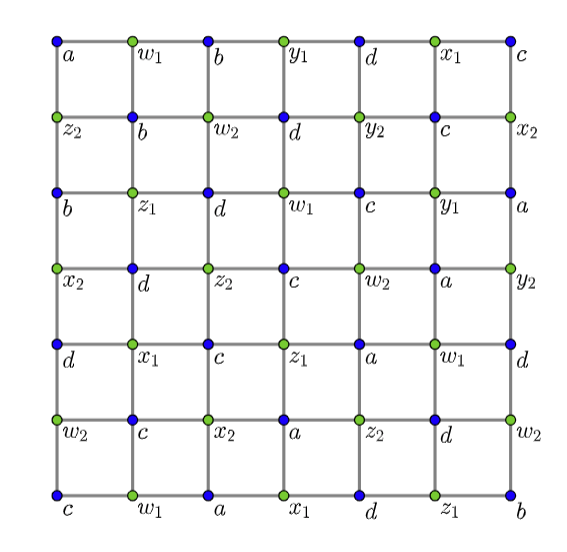}
		\caption{}
	\end{minipage}
\end{figure}

We now divide colors $x, y, z, w$ into $x_1, x_2, y_1, y_2, z_1, z_2, w_1, w_2$ by parity of $l_1$. Let
\[f(v)=
\left\{
\begin{array}{lcl}
	f_0(v),&\text{if }f_0(v)\in\{a,b,c,d\},\\
	x_1,y_1,z_1,w_1,&\text {if } f_0(v)=x,y,z,w\text{ respectively, }l_1(v) \text{ is odd},\\
	x_2,y_2,z_2,w_2,&\text{ if } f_0(v)=x,y,z,w\text{ respectively, }l_1(v) \text{ is even}.
\end{array}
\right.
\]
Then $f$ is a nonrepetitive coloring of $P_n\square P_n$ using 12 colors (Figure 8).
\end{proof}

Next, we briefly talk about the high-dimensional case, which is fundamentally similar to the 2-dimension approach, but with certain distinctions in the details.

\begin{theorem}
$\pi(P_n\square P_n\square P_n)\le28$ for all $n$.
\end{theorem}

\begin{proof}
Let $m$ be a sufficiently large positive integer. Consider $f_{W,m}^{(1)}:[m]\to\{a,b,c,d\}$, $f_{W,m}^{(2)}:[m]\to\{x,y,z,w\}$ and $f_{W,m}^{(3)}:[m]\to\{X,Y,Z,W\}$ in the proof of Theorem 1.6. Just like in Definition 2.1, we label the vertices of $P_n\square P_n\square P_n$ sequentially from $v_{(1,1,1)}$ to $v_{(n,n,n)}$, and define the location function $l_1$, $l_2$ and $l_3$.

We call the set $H_r=\{v\in V((P_n\square P_n\square P_n):l_1(v)+l_2(v)+l_3(v)=r\}$ a $left$ $plane$. Note that every left plane is an independent set. For all fixed vertex $v_0$, we call the set $L_{v_0}=\{v\in V(P_n\square P_n\square P_n):l_1(v)-l_1(v_0)=l_2(v)-l_2(v_0)=l_3(v)-l_3(v_0)\}$ a $right$ $line$ passing through $v_0$.

Based on the parity of the index $r=l_1(v)+l_2(v)+l_3(v)$, we classify the vertices $v$ of graph $P_n\square P_n\square P_n$ into two types. We then define a coloring $f_0$ of $P_n\square P_n\square P_n$ as follows:

If $l_1(v)+l_2(v)+l_3(v)=r$ is odd, then we let $f_0(v)=f_{W,m}^{(1)}(r)$. Note that all vertices on each left plane of odd $r$ share the same color.

If $l_1(v)+l_2(v)+l_3(v)=r$ is even, then we apply a 16-coloring scheme to ensure that all vertices lying on each given right line are assigned a common color. Below are the details of this coloring scheme.

We consider an arbitrary right plane $H=H_r$. Let $H^2$ be the induced subgraph of $V(H)$ in $(P_n\square P_n\square P_n)^2$, then $H^2$ is a $regular$ $triangular$ $tiling$ $graph$ (Figure 9) of some size, which is a subgraph of $P_m\boxtimes P_m$.

\begin{figure}[h] 
	\centering
	\includegraphics[width=0.25\linewidth]{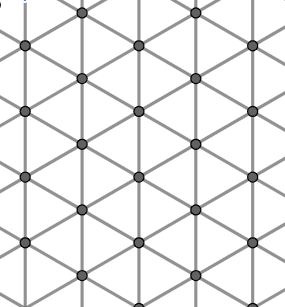}
	\caption{}
\end{figure}

So we can utilize product coloring $(f_2,f_3)$ to provide a walk-nonrepetitive 16-coloring for $P_m\boxtimes P_m$, and $H^2$ inherits this coloring. Since $m$ is sufficiently large, we can translate the coloring pattern along the direction of right lines to ensure that each left plane with an even index $r$ has a 16-coloring that is walk-nonrepetitive on its distance-2 graph, and each right line intersects all left planes with even indices $r$ at vertices of the same color. (See Figure 10. The dashed lines indicate the edges in the distance-2 graph.)

\begin{figure}[h] 
	\centering
	\begin{minipage}[h]{0.48\textwidth}
		\centering
		\includegraphics[width=\textwidth]{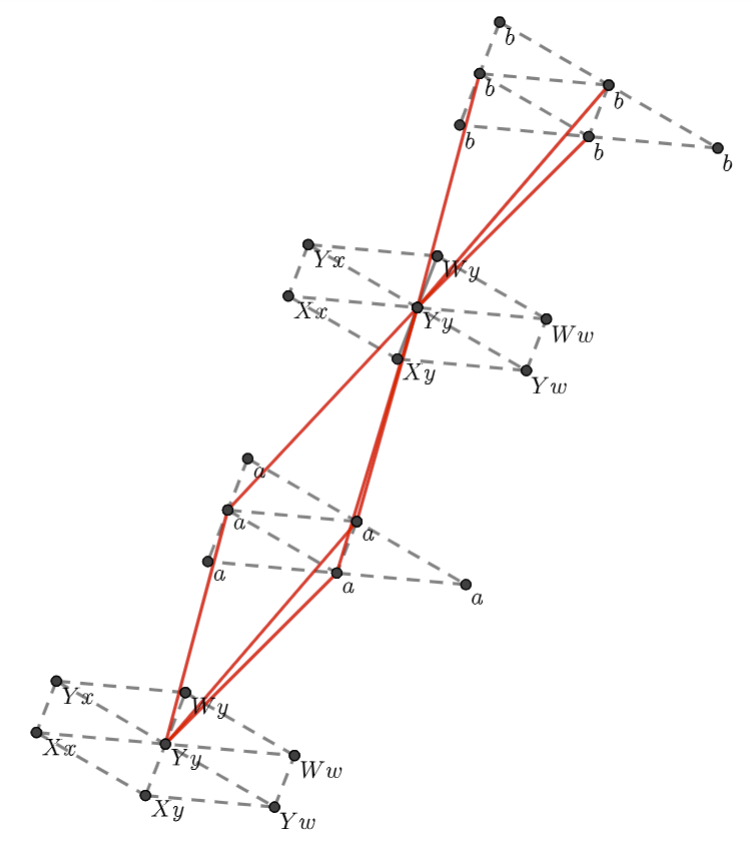}
		\caption{}
	\end{minipage}\hfill
	\begin{minipage}[h]{0.48\textwidth}
		\centering
		\includegraphics[width=\textwidth]{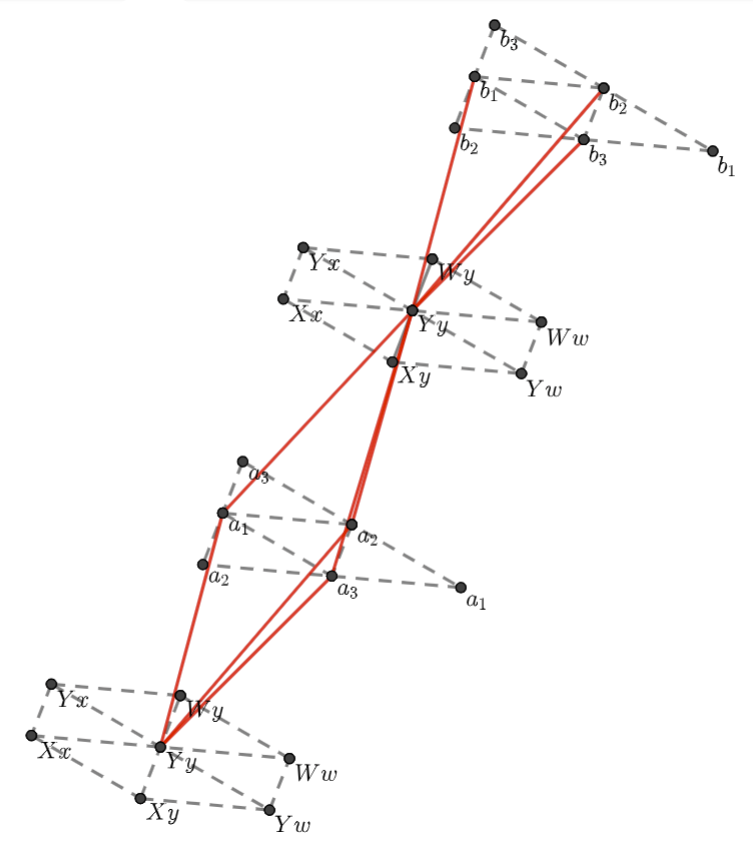}
		\caption{}
	\end{minipage}
\end{figure}	

Like the proof of Theorem 2.4, all repetitive paths under coloring $f_0$ can be \lq confined\rq\ within a small region (the red edges in Figure 10). By dividing the colors $a,$ $b,$ $c,$ $d$ into $a_i,$ $b_i,$ $c_i,$ $d_i,$ $i=1,2,3$, we destroy all repetitive paths (Figure 11). At this point, the total number of colors we use is 16 + 3×4 = 28.
\end{proof}

Why don't we study the higher-dimensional cases? In fact, our method cannot yield better results in grids with high dimensions. When constructing a nonrepetitive coloring of an $s$-dimensional grid, we need to utilize a walk-nonrepetitive coloring for an $(s-1)$-dimensional grid, which requires $4^{s-1}$ colors. However, the maximum degree of an $s$-dimensional grid is $2s-2$. In \cite{alon2002nonrepetitive}, Alon et al. proved that $\pi(G)=O(\Delta^2)$ where $\Delta$ denotes the maximum degree of $G$. This bound is superior for high-dimensional grids.

\newpage
\section{Lower bound}
Since $C_{10}$ is a subgraph of $P_n\square P_n$ and $\pi(C_{10})=4$ (see \cite{currie2002there} or \cite{WoodDavidR2021NGC}), it follows that $\pi(P_n\square P_n)\ge4$ for large $n$. In this section we prove that $\pi(P_n\square P_n)\ge5$ for large $n$.

For the sake of convenience, we express the (partial) coloring of the grid in the form of a matrix. For a grid $P_m\square P_n$ with a coloring, the $(i,j)$ element of the corresponding matrix is the color of the vertex $v_{(i,j)}$. For instance, the corresponding matrix of the coloring in Figure 12 is 

\[  
\left[  
\begin{array}{ccc}  
	a & c &   \\  
	  & d & b \\  
	a &   &    
\end{array}  
\right]  
.\]

\begin{figure}[h] 
	\centering
	\includegraphics[width=0.2\linewidth]{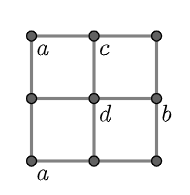}
	\caption{}
\end{figure}

\begin{theorem}
$\pi(P_n\square P_n)\ge5$ for large $n$.
\end{theorem}

\begin{proof}
Let $n$ be a sufficiently large positive integer. Suppose $P_n\square P_n$ has a nonrepetitive 4-coloring $f:V(P_n\square P_n)\to\{a,b,c,d\}$. We proceed to construct a contradiction in three steps as follows.

Note that we only have 4 colors available to use. When a local coloring is determined, in order to ensure that no repetitive path occurs, it can be deduced that certain vertices around it must be of certain colors. We use long right arrows between matrices to represent such deductions. For example, the following notation means that when the local coloring is described by the matrix on the left, the color of the vertex at the bottom left corner can only be $c$.

\[  
\begin{array}{ccc}  
	\begin{bmatrix}  
		  & a \\  
		a & d \\
		  & b
	\end{bmatrix}  
	& \longrightarrow &  
	\begin{bmatrix}  
		  & a \\  
		a & d \\
		c & b
	\end{bmatrix}   
	
\end{array}  
.\]  

Step 1. There must be some $C_4$ (as a subgraph of $P_n\square P_n$) that uses only 3 colors.

If every $C_4$ uses 4 colors, we only need to discuss the following two cases:

(1.1).\[  
\begin{array}{ccc}  
	\begin{bmatrix}  
		  &   & a \\  
		a & b &  \\
		d & c &  
	\end{bmatrix}  
	& \longrightarrow &  
	\begin{bmatrix}  
		  & c & a\\  
		a & b & d\\
		d & c &  
	\end{bmatrix}   
	\quad \longrightarrow \quad  
	\begin{bmatrix}  
	   d & c & a\\  
       a & b & d\\
       d & c & ! 
	\end{bmatrix}  
\end{array}  
.\]   

The exclamation mark $(!)$ indicates that regardless of which color is placed in this position, it will either result in a repetitive path or contradict the assumption.

(1.2).\[  
\begin{array}{ccc}  
	\begin{bmatrix}  
	   	  &   &  \\  
		a & b & a \\
		d & c &  
	\end{bmatrix}  
	& \longrightarrow &  
	\begin{bmatrix}  
	  	  &  & \\  
		a & b & a\\
		d & c & d
	\end{bmatrix}   
	\quad \longrightarrow \quad  

\end{array}  
\]   

(1.2.1).\[  
\begin{array}{ccc}  
	\begin{bmatrix}  
		d &   &  \\  
		a & b & a \\
		d & c & d
	\end{bmatrix}  
	& \longrightarrow &  
	\begin{bmatrix}  
		d & c & ! \\  
		a & b & a \\
		d & c & d
	\end{bmatrix}   
\end{array}  
;\] 

(1.2.2).\[  
\begin{array}{ccc}  
	\begin{bmatrix}  
		c &   &  \\  
		a & b & a \\
		d & c & d
	\end{bmatrix}  
	& \longrightarrow &  
	\begin{bmatrix}  
		c & d & c \\  
		a & b & a \\
		d & c & d
	\end{bmatrix}  
	\quad \longrightarrow \quad  
	\begin{bmatrix}  
		c & d & c & b \\  
		a & b & a & ! \\
		d & c & d &  
	\end{bmatrix} 
\end{array}  
.\] 

Now we have proven that there always exists a $C_4$ that is colored using only 3 colors. Without loss of generality, we assume that the coloring pattern of this $C_4$ corresponds to matrix 

\[  
\left[  
\begin{array}{ccc}  
	a & b \\  
	c & a   
\end{array}  
\right]  
.\]

Step 2. The 12 vertices surrounding such a $C_4$ cannot be colored with color $a$, that is, all elements marked with $X$ in following matrix cannot be $a$.

\[  
\left[  
\begin{array}{cccc}  
	X & X & X & X \\
	X & a & b & X \\  
	X & c & a & X \\
	X & X & X & X  
\end{array}  
\right]
.\]

By symmetry, we only need to discuss the following three cases:

(2.1).\[  
\begin{array}{cccc}  
	\begin{bmatrix}  
	  &   &   & a \\
	  & a & b &   \\  
	  & c & a &   \\
	  &   &   &    
	\end{bmatrix}  
	& \longrightarrow &  
\end{array}  
\] 

(2.1.1).\[  
\begin{array}{cccc}  
	\begin{bmatrix}  
		&   & d & a \\
		& a & b & d \\  
		& c & a & ! \\
		&   &   &    
	\end{bmatrix}  

\end{array}  
;\] 

(2.1.2).\[  
\begin{array}{cccc}  
	\begin{bmatrix}  
		&   & c & a \\
		& a & b &  \\  
		& c & a &  \\
		&   &   &    
	\end{bmatrix}  
	& \longrightarrow &  
\end{array}  
\]

(2.1.2.1).\[  
\begin{array}{cccc}  
	\begin{bmatrix}  
		&   & c & a \\
		& a & b &  \\  
		& c & a &  \\
		& a &   &    
	\end{bmatrix}  
	& \longrightarrow &  
	\begin{bmatrix}  
		&   & c & a \\
	  d & a & b &   \\  
		& c & a &   \\
		& a & d &    
	\end{bmatrix}   
	\quad \longrightarrow \quad  
		\begin{bmatrix}  
		&   & c & a \\
	  d & a & b &   \\  
	  a & c & a &   \\
	  ! & a & d &    
	\end{bmatrix}   
\end{array}  
;\]

(2.1.2.2).\[  
\begin{array}{cccc}  
	\begin{bmatrix}  
		&   & c & a \\
		& a & b &  \\  
		& c & a &  \\
		& d & ! &    
	\end{bmatrix}  
\end{array}  
.\]

(2.2).\[  
\begin{array}{cccc}  
	\begin{bmatrix}  
		&   &   &   \\
		& a & b &   \\  
		& c & a &   \\
		&   &   & a  
	\end{bmatrix}  
	& \longrightarrow &  
	\begin{bmatrix}  
		&   &   &   \\
		& a & b &   \\  
		& c & a & ! \\
		&   & d & a  
	\end{bmatrix}  
\end{array}  
.\]

(2.3).\[  
\begin{array}{cccc}  
	\begin{bmatrix}  
		&   & a &   \\
		& a & b &   \\  
		& c & a &   \\
		&   &   &  
	\end{bmatrix}  
	& \longrightarrow &  
	\begin{bmatrix}  
		& d & a &   \\
	  ! & a & b &   \\  
		& c & a &   \\
		&   &   &   
	\end{bmatrix}  
\end{array}  
.\] 

This indicates that the elements marked with $X$ in the matrix can only be $b,$ $c,$ or $d$. 

Step 3. A simple observation.
\[  
\begin{array}{cccc}  
	\begin{bmatrix}  
		& Y & Y & Y \\
		& a & b & Y \\  
		& c & a & Y \\
		&   &   &    
	\end{bmatrix}    
\end{array}  
.\] 

Based on the fact that $\pi(P_5)=3$, we can make a simple observation: among the elements marked with $Y$ in above matrix, $b$ must appear. However, to avoid the repetitive path, $b$ can only be placed in the top-right corner. Therefore, we only need to derive a contradiction from the following case.

\[  
\begin{array}{cccc}  
	\begin{bmatrix}  
		&   &   & b \\
		& a & b &   \\  
		& c & a &   \\
		&   &   &    
	\end{bmatrix}  
	& \longrightarrow &  
	(\text{By symmetry})
	\begin{bmatrix}  
		&   & c & b \\
		& a & b & d \\  
		& c & a & ! \\
		&   &   &    
	\end{bmatrix}  
\end{array}  
.\] 

This completes the proof.
\end{proof}

\section{The Cartesian product of complete graphs}

In Open Problem 3.28 of \cite{WoodDavidR2021NGC}, the nonrepetitive coloring of $K_n\square K_n$ plays an important role in studying the $edge$-$nonrepetitive$ $coloring$ of $K_n$. Below we give a coloring to prove that $\pi(K_n\square K_n)$ is at most $(\frac{1}{2}+o(1))n^2$.

\begin{theorem}
	For even $n\ge 4,$ $\pi(K_n\square K_n)\le\frac{n^2}{2}.$
\end{theorem}

\begin{proof}
	\begin{figure}[h] 
		\centering
		\includegraphics[width=0.4\linewidth]{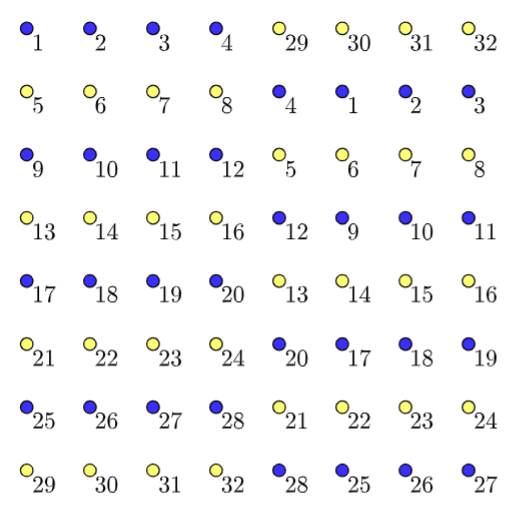}
		\caption{}
	\end{figure}
Label the vertices of $K_n\square K_n$ sequentially from $v_{(1,1)}$ to $v_{(n,n)}$, where $v_{(i,j)}\sim v_{(i',j')}$ if and only if $i=i'$ and $j\not=j'$ or $j=j'$ and $i\not=i'$. See Figure 13. Each row or each column induces a complete graph. Now we design a coloring $f:V(K_n\square K_n)\to[\frac{n^2}{2}]$ as follows:

For $j\le\frac{n}{2}$, let $f(v_{(i,j)})=\frac{n}{2}(i-1)+j$.

For $\frac{n}{2}\le j\le n$, we give different coloring schemes based on the parity of $i$.

When $i$ is odd, let
\[f(v_{(i,j)})=
\left\{
\begin{array}{lcl}
	f(v_{(i-1,j-\frac{n}{2})}),&i\not=1;\\
	f(v_{(n,j-\frac{n}{2})}),&i=1.\
\end{array}
\right.
\]

When $i$ is even, let
\[f(v_{(i,j)})=
\left\{
\begin{array}{lcl}
	f(v_{(i-1,j-\frac{n}{2}-1)}),&j\not=\frac{n}{2}+1;\\
	f(v_{(i-1,\frac{n}{2})}),&j=\frac{n}{2}+1.\
\end{array}
\right.
\]
Then $f$ is a nonrepetitive coloring of $K_n\square K_n$. This verification is straightforward, and we can illustrate it intuitively with Figure 13. We partition the vertices of the graph into two distinct perspectives: first, into left and right parts, and second, into blue and yellow types (note that blue and yellow here are not colors, but merely labels for types). It is noted that at this point, for any two disjoint edges, if they have the same starting color and ending color, respectively, then their vertices must belong to the same type and are located respectively in the left part and the right part. Consequently, under this coloring, there is no repetitive path.
\end{proof}

As a straightforward corollary, $\pi(K_n\square K_n)\le(\frac{1}{2}+o(1))n^2$ when $n\to\infty$.

\section{Discussions}

\begin{question}
Let $G$ be a regular triangular tiling graph (Figure 9) of any size. Is that $\pi(G)<16?$
\end{question}

The regular triangular tiling graph is a type of graph that lies between $P_n\square P_n$ and $P_n\boxtimes P_n$, and it is also a planar graph. Research on this question will directly contribute to the nonrepetitive coloring of planar graphs.

\begin{question}
Is that $\pi(P_n\square P_n)>5$ for large $n?$
\end{question}

In Section 3, we have proven that $\pi(P_n\square P_n)\ge5$ by analyzing some very small local structures. Can this approach be further developed or refined? Perhaps computers can help.

\begin{question}
Is there any constant $c>0$ such that $\pi(K_n\square K_n)\ge cn^2?$
\end{question}

In the previous text, we have mentioned that $\pi(G) = O(\Delta^2)$. For the lower bound, Alon et al. proved in \cite{alon2002nonrepetitive} using probabilistic methods that for all $\Delta$, there exists a graph $G$ with maximum degree $\Delta$ such that $\pi(G) \ge \frac{\Delta^2}{\log\Delta}$. Since the maximum degree of $K_n \square K_n$ is $2n-2$, if the answer to Question 5.3 is yes, then we would be able to improve the lower bound to the same order of magnitude as the upper bound.

\section{Acknowledgments}
Thanks to Peter Bradshaw, Hehui Wu, Qiqin Xie, Ningyuan Yang and Wentao Zhang for the early discussion of this project and the helpful comments.

\bibliographystyle{plain}
\bibliography{reference.bib}

\end{document}